\documentclass[11pt]{amsart}
\usepackage[utf8]{inputenc}

\usepackage{enumerate}
\usepackage{amssymb}
\usepackage{graphicx}
\usepackage{amscd, color}
\usepackage{amsmath}
\usepackage{amsfonts}
\usepackage[english]{babel}
\usepackage{epsfig}
\usepackage{color}

\usepackage{tikz-cd}

\numberwithin{equation}{section}

\newtheorem{theorem}{Theorem}[section]
\newtheorem{lemma}[theorem]{Lemma}
\newtheorem{proposition}[theorem]{Proposition}
\newtheorem{corollary}[theorem]{Corollary}
\theoremstyle{definition}
\newtheorem{definition}[theorem]{Definition}
\newtheorem{example}[theorem]{Example}
\newtheorem{question}[theorem]{Question}

\theoremstyle{remark}
\newtheorem{remark}[theorem]{Remark}

\numberwithin{equation}{section}

\DeclareMathOperator{\Aff}{Aff}
\DeclareMathOperator{\Sim}{S}

\DeclareMathOperator{\diam}{diam}

\begin{document}

\title{Some generalizations on affine invariant points}

\author{Natalia Jonard-P\'erez}


\thanks{The author has been partially supported by grants IN115819 (PAPIIT, UNAM, M\'exico) and 252849 (CONACYT-SEP, M\'exico)}

\address{Departamento de  Matem\'aticas,
Facultad de Ciencias, Universidad Nacional Aut\'onoma de M\'exico, 04510 Ciudad de M\'exico, M\'exico.}

\email{nat@ciencias.unam.mx}

\begin{abstract}

In this note we prove a more general (and topological) version of  Gr\"unbaum's conjecture about affine invariant points. As an application of our result we show that, if we consider the action of the group of similarities,  Gr\"unbaum's conjecture remains valid in other families of convex sets (not
necessarily convex bodies).
\end{abstract}

\maketitle

\section{Introduction}
By a convex body $K\subset\mathbb R^n$ we mean a compact convex subset of $\mathbb R^n$ with a non empty interior. The set of all convex bodies of $\mathbb R^n$ is denoted by $\mathcal K^n_0$. 
We denote by $\Aff(n)$ the group of all non-singular affine maps of $\mathbb R^n$. Namely, $g\in \Aff(n)$ if and only if there exists an invertible linear map $\sigma:\mathbb R^n\to \mathbb R^n$, and a vector $v\in\mathbb R^n$ such that 
$$g(x)=v+\sigma(x),\quad \text{for every }x\in\mathbb R^n.$$

If we equip $\mathcal K^n_0$ with the Hausdorff metric, the natural action of $\Aff(n)$ on $\mathbb R^n$ induces a continuous action of $\Aff(n)$ on $\mathcal K^n_0$ by  the formula
\begin{equation*}
(g, K)\longmapsto gK, \quad
gK=\{g(x)\mid x\in K\}
\end{equation*}
(see \cite{NataliaFM} for more details about this action). 

In his seminal paper \cite{Grunbaum}, B. Gr\"unbaum introduced the notion of an affine invariant point. Namely, an \textit{affine invariant point} is a continuous function $p:\mathcal K_0^n\to \mathbb R^n$ satisfying
$$gp(K)=pg(K),\text{ for every }g\in\Aff(n),\text{ and }K\in\mathcal K_0^n.$$

The centroid, the center of John's ellipsoid or the center of L\"owner's ellipsoid are examples of affine invariant points (see e.g. \cite[3.3 and 3.4]{Grunbaum}). 

We denote by $\mathfrak{P}_n$ the set of all affine invariant points of $\mathcal K_0^n$. 
For every $K\in\mathcal K_0^n$, let $\mathfrak{P}_n(K)$ be the set of all $x\in\mathbb R^n$ such that $x=p(K)$ for some $p\in\mathfrak{P}_n$,  and consider the set
$$\mathfrak{F}_n(K)=\{x\in\mathbb R^n\mid gx=x \text{ for every }g\in\Aff(n)\text{ such that }gK=K\}.$$
For a convex body $K\in\mathcal K^n_0$, it is not difficult to see that  $\mathfrak{P}_n(K)\subset\mathfrak{F}_n(K)$. In \cite{Grunbaum}, Gr\"unbaum conjectured that $\mathfrak{P}_n(K)=\mathfrak{F}_n(K)$, and in the following 50 years several partial results were obtained (see, e.g., \cite{Kuchment, Kuchment E, Meyer}). The first complete proof of this problem was obtained by  O. Mordhorst in \cite{Olaf}. 
In \cite{Iurchenko}, I. Iurchenko gave an alternative proof and in \cite{Jonard} the author gave a shorter proof of this theorem based merely on the topology of the action of $\Aff(n)$ on $\mathcal K^n_0$.

In the language of topological transformation groups, an affine invariant point is simply an $\Aff(n)$-equivariant continuous function from $\mathcal K_0^n$ into $\mathbb R^n$ (see Section ~\ref{s:preliminars}). 
On the other hand, if we denote by $\Aff(n)_K$ the stabilizer  of $K$, then the set 
$\mathfrak{F}_n(K)$  coincides with  the set of all $\Aff(n)_K$-fixed points of $\mathbb R^n$.  Namely, a point $x\in\mathbb R^n$ belongs to $\mathfrak{F}_n(K)$  if and only if $\Aff(n)_K\leq \Aff(n)_x$ ($\Aff(n)_x$  being the stabilizer of $x$).
If we translate Gr\"unbaum's situation to the language of topological transformation groups, we obtain  the following scenario: given a convex body $K\in \mathcal K_0^n$ and $x\in\mathbb R^n$ with $\Aff(n)_K\leq \Aff(n)_x$, Gr\"unbaum's conjecture states that there exists  a continuous and  $\Aff(n)$-equivariant function   $p:\mathcal K_0^n\to \mathbb R^n$ such that $p(K)=x$. 

The problem of finding an equivariant map between $G$-spaces is not an easy one. Some classical results such as the Borsuk-Ulam theorem need a very strong baggage of algebraic topology in order to be proved. 

In this note we use some classical results from the theory of topological transformation groups (\cite{Palais}) in order to  prove the following topological generalization of Gr\"unbaum's conjecture.

\begin{theorem}\label{t:main 2}
Let $G$ be a Lie group. Assume that $X$ is a proper $G$-space and $Y$ a $G$-equiconnected $G$-space such that  the set of all $G$-equivariant maps from $X$ to $Y$ is nonempty. If  $x\in X$ and $y\in Y$ satisfy $G_y\leq G_x$, then there exists a $G$-equivariant map $\psi:X\to Y$ such that $\psi(x)=y$.
 \end{theorem}

As usual, the word \textit{map} means  continuous function.  Since the action of the group $\Aff(n)$ on $\mathcal K^n_0$ is proper (see \cite[Theorem 3.3]{NataliaFM}) and $\mathbb R^n$ is $\Aff(n)$-equiconnected (see Section~\ref{s:equiconnected}), Theorem~\ref{t:main 2} directly implies Gr\"unbaum's conjecture.  

Definitions such as proper action or $G$-equiconnected space  is explained in Section~\ref{s:preliminars}. In Section~\ref{s:examples} we  show how Theorem~\ref{t:main 2} can be used to obtain several Gr\"unbaum-like results involving different families of compact convex subsets of $\mathbb R^n$ and different subgroups of affine transformations (Theorems~\ref{t:E(n)}, \ref{t:S(n) 1} and ~\ref{t:S(n) 2}).

\section{Preliminaries}\label{s:preliminars}

All basic terminology about topological spaces can be consulted in \cite{Engelking}.
In order to make our proofs  more comprehensive,   let us recall some basic notions about the theory of $G$-spaces. We refer the reader to the monographs \cite{Bredon} and \cite{Palais2} for a deeper understanding of this theory.

If $G$ is a topological group and $X$ is a $G$-space (namely, a topological space equipped with a continuous action of the group $G$), for any $x\in X$ we denote by $G_x$ the \textit{stabilizer}  or \textit{isotropy group} of $x$, i.e., $G_x=\{g\in G \mid gx=x\}$. For a subset $S\subset X$, the symbol $G(S)$ denotes the $G$-\textit{saturation} of $S$, i.e., $G(S)=\{gs\mid g\in G,\; s\in S\}.$ If $G(S)=S$ then we say that $S$ is a $G$-\textit{invariant} set (or, simply, \textit{invariant} set). In particular, $G(x)$ denotes the $G$-\textit{orbit} of $x$, i.e., $G(x)=\{gx\in X\mid g\in G\}$. The set consisting of all orbits of $X$ equipped with the quotient topology is  denoted by $X/G$ and is called the \textit{orbit space}. The quotient map $\pi:X\to X/G$ is called the \textit{orbit map} and it is always open.

For each subgroup $H\subset G$, the $H$-\textit{fixed point set} $(X)^H$ is the set $\{x\in X\mid H\leq G_x\}$. Clearly, $(X)^{H}$ is a closed subset of $X$.  Observe that in Gr\"unbaum's conjecture, the set $\mathfrak{F}_n(K)$ coincides with the $\Aff(n)_K$- fixed point set $(\mathbb R^n)^{\Aff(n)_K}$.

A  map $f:X\to Y$ between two $G$-spaces is called \textit{equivariant} or a $G$-\textit{map} if
$f(gx)=gf(x)$ for every $x\in X$ and $g\in G$. On the other hand, if $f(gx)=f(x)$ for every $x\in X$, we call $f$ an \textit{invariant} (or $G$-\textit{invariant}).
The reader may notice that an affine invariant point is just an $\Aff(n)$-equivariant map. In order to avoid any confusion, from now on we use exclusively the terminology stated in this section.   

Given $G$-spaces $X$ and $Y$, we denote by $C_G(X,Y)$ the set  of all $G$-equivariant maps. Notice that $C_G(X,Y)$ may be an empty set.  On the other hand, the set $\mathfrak{P}_n$ coincides precisely with the set $C_G(X,Y)$ where $G=\Aff(n)$, $X=\mathcal K^n_0$ and $Y=\mathbb R^n$.

The following folklore is used in the proof of our main result. 

\begin{lemma}\label{l:funcion de tychonoff} Let $X$ be a $G$-space, $x\in X$ an arbitrary point and $V\subset X$ a $G$-invariant neighborhood of the orbit $G(x)$.
\begin{enumerate}
\item If $X/G$ is regular then we can find an invariant neighborhood $U$ of $G(x)$ such that
$$G(x)\subset U\subset\overline U\subset V.$$
\item If  $X/G$ is Tychonoff then  we can find an invariant map 
$\lambda:X\to [0,1]$ such that $\lambda(y)=1$ and $\lambda (z)=0$ for every $y\in G(x)$ and $z\in X\setminus V$.
\end{enumerate}
\end{lemma}

The proof of Lemma~\ref{l:funcion de tychonoff} can be consulted in \cite[Lemma 2.2]{Jonard}.

\subsection{Cartan and Proper $G$-spaces}

If a non compact group acts continuously on a topological space, many pathologies may occur. This is why we  focus our attention in a specific kind of actions called proper actions.

\begin{definition}\label{d:cartan proper}
Let $G$ be a locally compact Hausdorff group and let $X$ be a Tychonoff $G$-space. 
\begin{enumerate}
    \item $X$ is a \textit{Cartan} $G$-space if every point  $x\in X$ has a  neighborhood $V$ with the property that $\{g\in G: gV\cap V\neq\emptyset\}$ has compact closure in $G$. A set $V$ satisfying this property is called a  \textit{thin} set.
    \item   $X$ is a  \textit{proper} $G$-space (in the sense of Palais \cite{Palais}) if it has an open cover consisting of, so called,  small sets.
 A set  $S\subset X$ is called \textit{small} if  any  point $x\in X$ has  a neighborhood $V$  such that the set $\langle S, V\rangle=\{g\in G\mid gS\cap V\neq\emptyset\}$  has compact closure in $G$.
\end{enumerate}

\end{definition}

The following result shows the relation between Cartan and proper $G$-spaces.

\begin{proposition}(\cite[Corollary 1, p. 303]{Palais})\label{p:cartan propio}
Let $G$ be a locally compact Hausdorff group and let $X$ be a Tychonoff $G$-space. Then $X$ is a proper $G$-space if and only if $X$ is a Cartan $G$-space and the orbit space $X/G$ is regular. In particular, every proper $G$-space is Cartan. 
\end{proposition}

In the following theorem we summarize some useful properties about Cartan spaces.
The proof of these and other related results can be consulted in \cite{Palais}

\begin{theorem}\label{t:properties} Let $X$ be a Tychonoff space. If $X$ is a Cartan $G$-space then the following conditions hold:
\begin{enumerate}
\item For every $x\in X$ the map $g\mapsto gx$ is an open map of $G$ onto $G(x)$.
\item For every $x\in X$, the orbit $G(x)$ is closed, the stabilizer $G_x$ is compact and $G(x)$ is $G$-homeomorphic to $G/G_x$.

\item If additionally $X$ is a proper $G$-space, then the orbit space $X/G$ is Tychonoff.
\end{enumerate}

\end{theorem}

Another important result about Cartan $G$-spaces, is the following.

\begin{theorem}\label{t:rebanada}\cite[Corollary 1, p. 313]{Palais}
Let $G$ be  a Lie group, and $X$ a Cartan $G$-space. Every orbit  $G(x)$ of $X$  is an equivariant retract of an invariant neighborhood of $G(x)$. Namely, for every $x\in X$ there exists an invariant neighborhood $U$ of $G(x)$ and $r:U\to G(x)$ a $G$-equivariant retraction. 
\end{theorem}

Recall that a subset $A$ of a topological space $X$ is called a \textit{retract} of $X$ if there is a continuous map $r:X\to A$ such that $r(a)=a$ for every $a\in A$.  In this case the map $r$ is called a \textit{retraction}.

Given two $G$-spaces $X$ and $Y$, and a pair $(x,y)\in X\times Y$,  the function $f:G(x)\to G(y)$ given  by $f(gx)=gy$ is well-defined if and only if $G_x\leq G_y$. However, even if $f$ is well-defined, the function $f$ is not always continuous. This is why we are interested in the  following easy lemma. We include its proof for the sake of completeness. 

\begin{lemma}\label{l:orbita a orbita}
Let $G$ be a topological group and  $X$ and $Y$ $G$-spaces.
Assume that $x\in X$ and $y\in Y$ satisfy that $G_x\leq G_y$. If the map $\theta_x:G\to G(x)$ given by $\theta_x(g)= gx$ is open, then the map
$f:G(x)\to G(y)$ given by $f(gx)=gy$ is continuous. 
In particular, if $X$ is a Cartan $G$-space (or a proper $G$-space), then $f$ is continuous.
\end{lemma}

\begin{proof}
Since the action of $G$ on $Y$ is continuous, for any $g\in G$ and any open neighborhood $U$ of $gy$ in $Y$, we can find  an open neighborhood $V$ of $g$ in  $G$, such that $$Vy:=\{hy\in Y: h\in V\}\subset U.$$
By hypothesis, the set $Vx:=\{hx \in X : h\in V \}$ is open in the orbit $G(x)$ and contains the point $gx$. Thus
$f(Vx)=Vy\subset U$ and therefore  $f$ is continuous in the point $gx$.

For the last part of the lemma, we simply use the fact that the map $\theta_x$ is always open in the category of Cartan $G$-spaces (Theorem~\ref{t:properties}).
\end{proof}

\subsection{G-equiconnected spaces}\label{s:equiconnected}

A $G$-equiconnected space is the equivariant version of the well-known notion of an equiconnected  space (see e.g. \cite{equiconnected}).
\begin{definition}\label{d:G equiconnected}
A topological  $G$-space $X$ is \textit{$G$-equiconnected} if there exists a map $h:X\times X\times[0,1]\to X$ satisfying the following conditions
\begin{enumerate}
    \item $h(x,y,0)=x$ and $h(x,y,1)=y$ for every $(x,y)\in X\times X$.
     \item $h(x,x,t)=x$ for every $x\in X$ and $t\in [0,1]$.
      \item $h(gx,gy,t)=gh(x,y,t)$ for each $g\in G$, $(x,y)\in X\times X$ and $t\in [0,1]$.
\end{enumerate}
In this case we say that the map $h$ has the \textit{$G$-equivariant connecting property} (or, simply, $h$ has the \textit{$G$-connecting property}).
\end{definition}

A very simple example of a $G$-equiconnected space is the euclidean space $\mathbb R^n$ equipped with the natural action of any subgroup of affine transformations $G\leq \Aff(n)$. Observe that in this case the map $h:\mathbb R^n\times\mathbb R^n\times[0,1]\to \mathbb R^n$
given by 
\begin{equation}\label{f:equiconnected}
 h(x,y,t)=(1-t)x+ty
\end{equation}
 has the $G$-equivariant connecting property. A more general example is any $G$-invariant convex subset of a topological vector $G$-space, where $G$ is a group of affine transformations. In this case the equiconnected map is defined as in formula (\ref{f:equiconnected}).
 
The Minkowski sum also defines a $G$-equiconnected structure in the space of all compact convex subsets of $\mathbb R^n$ (see Remark~\ref{r:minkowski sum equiconnected} below).
Another example of a $G$-equiconnected space  is any   $G$-space admitting a $G$-convex structure (see \cite{Convex strucuture}).

\subsection{Families of compact convex subsets of $\mathbb R^n$}

For a fixed $n\in\mathbb N$, we denote by $\mathcal K^n$ the family of all  compact convex subsets of $\mathbb R^n$.
We  equip $\mathcal K^n$ with the well-known Hausdorff distance
defined by the formula
$$d_H(A,B)=\max\left\{ \sup\limits_{b\in B}d(b,A),~~~~~\sup\limits_{a\in A} d(a, B)\right\},$$
where $d$ is the standard euclidean metric on $\mathbb R^n$. The symbol $\|\cdot\|$ is always  used to denote the euclidean norm in $\mathbb R^n$.

For any  $x\in \mathbb R^n$,  $K\in \mathcal K^n$ and $\varepsilon>0$ we  use the following notations:
$$B(x,\varepsilon):=\{y\in\mathbb R^n\mid d(y,x)<\varepsilon\},$$
$$N(K,\varepsilon):=\{y\in\mathbb R^n\mid d(y,K)<\varepsilon\}.$$

We recall that $d_H(A,B)<\varepsilon$ if and only if $A\subset N(B,\varepsilon)$ and $B\subset N(A,\varepsilon)$.

In section~\ref{s:examples}, we are interested in the following families of convex sets

$$\mathcal K^n_{j^{-}}:=\{K\in\mathcal K^n\mid \dim (K)\leq j\}$$
    $$\mathcal K^n_{ j^+}:=\{K\in\mathcal K^n\mid \dim (K)\geq j\}$$
where $j\in\{0,\dots, n\}$ and $\dim (K)$ denotes the dimension of $K$, namely, the dimension of the smallest affine subspace of $\mathbb R^n$ containing $K$. Observe that with respect to the previous notation, we have that $\mathcal K^n_0=\mathcal K^n_{n^+}$ and $\mathcal K^n=\mathcal K^n_{0^+}=\mathcal K^{n}_{n^-}$.

On the other hand, the family  $\mathcal K^n_{ 1^+}$ is precisely  the family of all non degenerated compact convex subsets of $\mathbb R^n$. Some topological and dynamical  properties of $\mathcal K^n_{ 1^+}$ were studied in \cite{Jonard-Merino}.  In particular, it was proved in that the orbit space $\mathcal K^n_{ 1^+}/\Sim(n)$ is a compact metric space homeomorphic to the Banach-Mazur compactum $BM(n)$, where $\Sim(n)$ stands for the \textit{group of all similarities} of $\mathbb R^n$ (see \cite[Corollary 4.7]{Jonard-Merino}).

The following remark will be used later.
\begin{remark}\label{r:minkowski sum equiconnected}
For any subgroup $G\leq \Aff(n)$ and any $0\leq j \leq n$, the space $\mathcal K^n_{ j^+}$ is $G$-equiconnected. The $G$-connecting map  $h:\mathcal K^n_{ j^+}\times \mathcal K^n_{ j^+}\times [0,1]\to \mathcal K^n_{ j^+}$ is the one defined by means of the Minkowski sum,
\begin{equation*}
    h(A,B,t)=(1-t)A+tB, \;\; A, B\in\mathcal K^n_{j^{+}}, \; \; t\in[0,1].
\end{equation*}

Since $\dim \left((1-t)A+tB\right)\geq \min\{\dim (A), \dim (B)\}\geq j$, the map $h$ is well-defined. 
\end{remark}

\section{Main results}

\begin{theorem}\label{t:main}
Let $X$ and $Y$ be $G$-spaces. Assume that $\{x_1,\dots, x_n\}\subset X$ and $\{y_1,\dots, y_n\}\subset Y$ are finite sets of points satisfying, for each $i\in \{1,\dots, n\}$, the following conditions
\begin{itemize}

    \item[(A)]The map $\theta_{x_i}:G\to G(x_i)$, given by $\theta_{x_i}(g)=gx_i$, is open.
    \item[(B)] There exists an invariant open neighborhood $U_i\subset X$ of $G(x_i)$ and a $G$-equivariant retraction $r_i:U_i\to G(x_i)$.
    \item[(C)] $G_{x_i}\leq G_{y_i}$.
    \item[(D)] $G(x_i)\cap G(x_j)=\emptyset$ whenever $i\neq j$.
\end{itemize}
Then, if $Y$ is $G$-equiconnected, $X/G$ is  Tychonoff and $C_G(X,Y)$ is nonempty, there exists a continuous  $G$-equivariant map $\psi\in C_G(X,Y)$ such that
$\psi(x_i)=y_i$ for every $i\in\{1,\dots ,n\}$.
\end{theorem}

\begin{proof}

Since $X/G$ is Tychonoff and the orbits $\{G(x_1), \dots, G(x_n)\}$ are disjoint, there exist $W_1,\dots, W_n\subset X $ disjoint invariant open sets of $X$ such that $G(x_i)\subset W_i$, $i=1,\dots, n$. 
Let $U_1,\dots, U_n$ and $r_1,\dots ,r_n$ be the neighborhoods and the maps of  condition (B). For every $i\in \{1,\dots, n\}$ define  the set  $O_i=U_i\cap W_i$. Clearly each $O_i$ is an open invariant neighborhood of $G(x_i)$ and the restriction $q_i:=r_i|O_i:O_i\to G(x_i)$ is a $G$-equivariant retraction. 

By Lemma~\ref{l:orbita a orbita},  for every $i\in \{1,\dots, n\}$, the map $f_i:G(x_i)\to G(y_i)$ given by $f_i(gx_i)=gy_i$ is continuous and $G$-equivariant. Then the map $\widetilde f_i:O_i\to G(y_i)$ defined by $\widetilde f_i=f_i\circ q_i$ is continuous and $G$-equivariant too. 
Since $O_i\cap O_j=\emptyset$ if $i\neq j$, the map $f:\bigcup\limits_{i=1}^n O_i\to Y$ given by 
$$f(z)=\widetilde f_i(z),\; \; \text{ if }z\in O_i$$
is well-defined, continuous and $G$-equivariant.

Using again the fact that $X/G$ is a Tychonoff space, we can find open invariant neighborhoods $V_i\subset X$ and continuous invariant maps $\lambda_i:X\to [0,1]$  such that

\begin{itemize}
    \item [a)] $G(x_i)\subset V_i\subset \overline{V_i}\subset O_i$,
    \item [b)] $\lambda_i(G(x_i))=\{1\}$,
    \item [c)] $\lambda_i(X\setminus V_i)=\{0\}$,
\end{itemize}
 for every $i\in \{1,\dots, n\}$ (see Lemma~\ref{l:funcion de tychonoff}).
 Define $\lambda :X\to [0,1]$ by $\lambda (z)=\sum\limits_{i=1}^n\lambda_i(z)$, for every $z\in X$. Clearly $\lambda$ is continuous and invariant. Further, since all neighborhoods ${V_1,\dots ,V_n}$ are disjoint, the map $\lambda$ also satisfies that
 \begin{itemize}
     \item [d)] $\lambda (z)\in [0,1]$ for every $z\in X.$
     
     \item [e)] $\lambda(G(x_i))=\lambda_i(G(x_i))=\{1\}$ for each $i\in\{1,\dots, n\}$.
     \item [f)] $\lambda (z)=0$ if $z\in X\setminus \bigcup\limits_{i=1}^n V_i$.
 \end{itemize}

Now, since the set of $G$-maps is nonempty, we can pick an equivariant map $\varphi\in C_G(X,Y)$. Furthermore, since $Y$ is $G$-equiconnected, there exists a $G$-connecting map $h:Y\times Y\times[0,1]\to Y$.
To finish the proof, define the map $\psi:X\to Y$ by the formula

\begin{equation*}
\psi(z)=\begin{cases}
\varphi(z),&\text{if } z\in X\setminus  \bigcup\limits_{i=1}^n\overline{V}_i,\\
h(\varphi(z), f(z),\lambda (z))&\text{if }x\in \bigcup\limits_{i=1}^n O_i.
\end{cases}
\end{equation*}

Observe that if $z \in \left( X\setminus  \bigcup\limits_{i=1}^n\overline{V}_i \right)\cap \left (\bigcup\limits_{i=1}^n O_i\right)$, then $\lambda (z)=0 $ and therefore $h(\varphi(z), f(z),\lambda (z))=\varphi(z)$. This proves that $\psi$ is well-defined and  continuous. 

To see that $\psi $ is $G$-equivariant, first observe that $X\setminus  \bigcup_{i=1}^n\overline{V}_i$ and $\bigcup_{i=1}^n O_i$ are invariant sets. Thus, for every $g\in G$,  if $z\in X\setminus  \bigcup_{i=1}^n\overline{V}_i$ then $gz\in X\setminus  \bigcup_{i=1}^n\overline{V}_i$. In this case
$$\psi (gz)=\varphi(gz)=g\varphi(z)=g\psi(z).$$
On the other hand, if $z\in \bigcup_{i=1}^n O_i$, then $gz\in \bigcup_{i=1}^n O_i$ and therefore
\begin{align*}
\psi(gz)&=h(\varphi(gz), f(gz),\lambda (gz))&\\
&=h(g\varphi(z), gf(z),\lambda (z))\\
&=gh(\varphi(z), f(z),\lambda (z))=g\psi(z).
\end{align*}
This equality completes the proof.

\end{proof}

After the previous theorem, the  following definition arises naturally.

\begin{definition}
Let $X$ and $Y$ be $G$-spaces. We  say that the pair $(X,Y)$ has the \textit{$G$-finite extension property} if  for  all finite sets   $\{x_1,\dots, x_n\}\subset X $ and $\{y_1,\dots, y_n\}\subset Y$ satisfying 
\begin{enumerate}
    \item $G_{x_i}\leq G_{y_i}$, for every $i\in \{1,\dots, n\}$,
    \item $G(x_i)\cap G(x_j)=\emptyset$ if $i\neq j$,
\end{enumerate}  there exists a $G$-map $\psi:X\to Y$ such that
$\psi(x_i)=y_i$ for every $i\in\{1,\dots ,n\}$.
\end{definition}
The reader may compare this definition with the definition of a $G$-absolute extensor (see Section~\ref{s:final}).

In \cite[Corollary 4.1]{Olaf}, O. Mordhorst proved that the pair $(\mathcal K^n_0, \mathbb R^n)$ has the $\Aff(n)$-finite extension property, generalizing in this way  Gr\"unbaum's conjecture. 

If $G$ is a Lie group and $X$ is a proper $G$-space, conditions (A) and (B) of Theorem~\ref{t:main} are guaranteed (see Theorems~\ref{t:properties} and \ref{t:rebanada}).
Furthermore, in this case the orbit space $X/G$ is always Tychonoff (Theorem~\ref{t:properties}). 
All this allows us to conclude the following

\begin{corollary}\label{c:main}
Let $G$ be a Lie group acting properly on $X$. Suppose that $Y$ is a $G$-equiconnected space such that $C_G(X,Y)$ is nonempty. Then the pair $(X,Y)$ has the $G$-finite extension property.

\end{corollary}

Now, the proof of Theorem~\ref{t:main 2} is a direct consequence of Corollary~\ref{c:main}.

\begin{remark}

Given two $G$-spaces $X$ and $Y$, condition $C_G(X,Y)$ can be guaranteed in certain cases. For example,  if Y has a $G$-fixed point, namely, a point $y_0\in Y$ such that $gy_0=y_0$ for every $g\in G$, then the constant map $\varphi:X\to Y$ given by $\varphi(x)=y$ belongs to $C_G(X,Y)$.

On the other hand, a  necessary condition in order that 
$C_G(X,Y)$ be non empty, is that for every $x\in X$ there exists a point $y\in Y$ such that $G_x\leq G_y$. 
\end{remark}

\section{Examples and Applications}\label{s:examples}

Our main purpose is to generalize Gr\"unbaum's conjecture to different families of convex subsets of $\mathbb R^n$. 
\begin{example}\label{e:dim 1}
We consider the space $\mathcal K^n_{1^-}$ equipped with the action of the group $\Aff(n)$. Namely,  $\mathcal K^n_{1^-}$ is the space of all segments of $\mathbb R^n$. Observe that the action of the affine group $\Aff(n)$ on $\mathcal K^n_{1^{-}}$ is not proper (indeed, the isotropy group of the singleton $\{0\}$ is the general linear group which is not compact, therefore the action cannot be proper). Furthermore, the orbit space $\mathcal K^n_{ 1^{-}}/\Aff(n)$ is not even a $T_1$-space. In fact, $\mathcal K^n_{1^{-}}/\Aff(n)$ is  homeomorphic to the  Sierpi\'nski space $Z$. Recall that the \textit{Sierpi\'nski space} (also named \textit{the connected two-point set}) is the topological space $(Z,\tau)$ where  $Z=\{0,1\}$ and the topology $\tau$ is given by $\tau=\{Z, \emptyset, \{0\}\}$. In this case the homeomorphism $\eta: Z\to \mathcal K^n_{ 1^{-}}/\Aff(n)$ sends the point $0$ to the orbit made of all non degenerated segments, and the point $1$ to the orbit of all singletons. 

However, for every nontrivial  segment $[a,b]\in\mathcal K^n_{1^{-}}$ and every point 
 $y\in \mathbb R^n$ such that the $\Aff(n)$-stabilizer  of $[a,b]$ is contained in the $\Aff(n)$-stabilizer of $y$, there exists a $\Aff(n)$-equivariant map $\psi:\mathcal K^n_{1^-}\to \mathbb R^n$ such that $\psi ([a,b])=y$. Furthermore, the map $\psi$ is the only $\Aff(n)$-equivariant map from $\mathcal K^n_{1^{-}}$ to $\mathbb R^n$.
 \end{example}
 
 \begin{proof}
First observe that the point $y$ is the middle point of $[a,b]$. Indeed, if we denote by $\mathcal F$ the set of points of $\mathbb R^n$ which are fixed by any affine transformation fixing the segment $[a,b]$ (namely, $\mathcal F=(\mathbb R^n)^G$, where $G=\Aff(n)_{[a,b]}$), then $y\in \mathcal F$. Let us denote by $m$ the middle point of $[a,b]$ and let $\Pi$ be the hyperplane of $\mathbb R^n$ orthogonal to $[a,b]$ that passes through $m$.
 If $\rho:\mathbb R^n\to\mathbb R^n$ is the reflection on $\Pi$, then $\rho$ belongs to the stabilizer of $[a,b]$. Therefore $\mathcal F\subset \Pi$. On the other hand, if $\ell$ is the line defined by $[a,b]$, and $\sigma$ is any rotation around $\ell$, then $\sigma$ belongs to the stabilizer of $[a,b]$ and therefore $\mathcal F\subset \ell$. Observe that $\ell $ intersects $\Pi$ in $m$. Thus 
 $$y\in \mathcal F\subset \Pi\cap \ell=\{m\},$$
 as desired.
 
To complete the proof, simply define 
$\psi:\mathcal K^{n}_{1^-}\to \mathbb R^n$ as follows
$$\psi ([c,d])=\frac{1}{2}(c+d).$$
Clearly $\psi$ is a well-defined $\Aff(n)$-equivariant map. Since $y=m=\frac{1}{2}(a+b)$, the map $\psi$ is the desired one, and it is the only $\Aff(n)$-equivariant map from $\mathcal K^{n}_{ 1^-}$ to $\mathbb R^n$.

 \end{proof}

In the following example we will show that the condition of the action to be proper is essential. 

\begin{example}
Let $T$ be the triangle in $\mathbb R^2$ with vertices $(0,0)$, $(1,0)$ and $(0,1)$, and let $b\in\mathbb R^2$ be its centroid, namely $b=(1/3, 1/3)$.  Since the centroid map is an $\Aff(n)$-equivariant map from $\mathcal K^n_{0}$ into $\mathbb R^n$ (see, e.g. \cite{Maria}), it follows that the $\Aff(2)$-stabilizer of $T$ is contained in  the $\Aff(2)$-stabilizer of $b$. However, there is no $\Aff(2)$-equivariant map $\psi:\mathcal K^2\to \mathbb R^n$ such that $\psi (T)=b$. Indeed, if such a map exists, then by the previous example $\psi (I)=(0,1/2)$, where $I=\{0\}\times[0,1]$. On the other hand, the sequence $(T_n)_{n\in\mathbb N}$ converges to $I$, where $T_n$ is the triangle with vertices $\{(0,0), (1/n,0), (0,1)\}$. Since  $T_n$ lies in the orbit of $T$ and $\psi$ is equivariant, then $\psi (T_n)$ is the centroid of $T_n$, namely $\psi (T_n)=\left(\frac{1}{3n}, \frac{1}{3}\right)$. By  continuity of $\psi$ we have that
$$(0,1/2)=\psi (I)=\lim_{n\to \infty}\psi(T_n)=\lim_{n\to\infty}\left(\frac{1}{3n}, \frac{1}{3}\right)=(0, 1/3),$$
a contradiction.
\end{example}

In the previous example, the action of the group $\Aff(2)$ on $\mathcal K^2$ is not proper. On the other hand, the action of the group $\Sim (n)$  (the group of all  similarities of $\mathbb R^n$)  on $\mathcal K^n$ is not proper either. However, in this case  it is interesting  to notice that $C_{\Sim(n)}(\mathcal K^n, \mathbb R^n)$ is not empty and it contains some important selectors such as the Steiner point or the Chebyshev point (see, e.g. \cite[Chapter 12]{Maria}). 

In \cite{Dantzig} it was proved that for any connected and locally compact metric space $X$,  the group  of all isometries of $X$ acts properly on $X$. From this important theorem we can conclude that the action of the \textit{Euclidean group} $E(n)$ (i.e., the group of all isometries of $\mathbb R^n$ with respect to the euclidean norm) acts properly on $\mathcal K^n$. 

Finally, observe that every $\Sim(n)$-equivariant map $\varphi:\mathcal K^n\to\mathbb R^n$ is $E(n)$-equivariant and therefore
 $C_{E(n)}(\mathcal K^n, \mathbb R^n)$ is nonempty. 
All these observations in combination with  Corollary~\ref{c:main} yield the following Gr\"unbaum-like theorem.

\begin{theorem}\label{t:E(n)}
The pair $(\mathcal K^n, \mathbb R^n)$ has the $E(n)$-finite extension property. 
\end{theorem}

The identity map $1_{\mathcal K_0^n}:\mathcal K_0^n\to \mathcal K_0^n$ is an $\Aff(n)$-equivariant map, therefore $C_{\Aff(n)}(\mathcal K^n_0,\mathcal K^n_0)$ is non empty. Since the action of $\Aff(n)$ on $\mathcal K^n_0$ is proper, we can apply Corollary~\ref{c:main} to infer the following theorem.
\begin{theorem}\label{t:kn0 kn0}
The pair $(\mathcal K^n_0, \mathcal K^n_0)$ has the $\Aff(n)$-finite extension property. 
\end{theorem}

Before our last results   we recall that for every similarity $g\in\Sim(n)$ there exist unique $u\in \mathbb R^n$, $\lambda>0$ and $\sigma\in O(n)$ such that $g(x)=u+\lambda\sigma(x)$ for every $x\in\mathbb R^n$ (where $O(n)$ denotes the orthogonal group). Thus, as a topological space, $\Sim(n)$ is homeomorphic to the topological product $$\mathbb R^n\times (0,\infty)\times O(n)\cong\Sim(n).$$ This will be used in the the following lemma.

\begin{lemma}\label{l:Sim propio}
For every $i\in\{1,\dots, n\}$ the action of the group $\Sim(n)$ on $\mathcal K^n_{i^{+}}$ is proper.
\end{lemma}
\begin{proof}

First, observe that $\mathcal K^n_{i^{+}}$ is an $\Sim(n)$-invariant subspace of $\mathcal K^n_{1^{+}}$. Then, in order to prove that $\mathcal K^n_{i^{+}}$ is proper it is enough to prove that the action of $\Sim(n)$ on $\mathcal K^n_{1^{+}}$ is proper.

By \cite[Corollary 4.7]{Jonard-Merino} the orbit space $\mathcal K^n_{1^{+}}/\Sim(n)$ is homeomorphic to a  Banach-Mazur compactum, and therefore it is metrizable. In particular, $\mathcal K^n_{ 1^{+}}/\Sim(n)$ is regular. Thus, according to Proposition~\ref{p:cartan propio} it suffices to prove that $\mathcal K^n_{1^{+}}$ is a Cartan $\Sim(n)$-space. 

We consider an arbitrary element $A\in\mathcal K^n_{1^{+}} $. Since $\dim (A)\geq 1$, we infer that $m:=\diam (A)>0$, where
$$\diam (A)=\max\{\|a-a'\|: a, a'\in A\}.$$
Let $\delta>0$ be such that $m-2\delta>0$ and consider 
$$\mathcal O:=\{B\in\mathcal K^n_{ 1^{+}}\mid d_H(B, A)<\delta\},$$ the $\delta$-ball around $A$ with respect to the Hausdorff metric. 
We claim that $\mathcal O$ is a thin neighborhood of $A$. In order to see this, first observe that
if $B\in \mathcal O$, there exist $b_1,b_2\in B$ such that $\|b_1-b_2\|=\diam (B)$. Since $d_H(B, A)<\delta$ we can pick points $a_1,a_2\in A$ with the property that $\|a_j-b_j\|<\delta$ (with $j=1, 2$). Thus, 
$$\diam(B)=\|b_1-b_2\|\leq \|b_1-a_1\|+\|a_1-a_2\|+\|a_2-b_2\|<\diam (A)+2\delta.$$
Analogously we can prove that $\diam (A)<\diam (B)+2\delta$ and therefore we get that
\begin{equation}\label{d:diametros}
    m-2\delta=\diam(A)-2\delta<\diam(B)< \diam(A)+2\delta= m+2\delta.
\end{equation}

Now, our goal is to prove that the set  $\Gamma:=\{g\in \Sim(n)\mid g\mathcal O\cap \mathcal O\neq\emptyset \}$ has compact closure in $\Sim(n)$. 
Let $M>0$ be such that for every $C\in\mathcal O$ and every $x\in C$, 
\begin{equation}\label{d:cota norma}
    \|x\|\leq M.
\end{equation}

Take $g\in \Gamma$ and let $B\in\mathcal O$ be such that $gB\in\mathcal O$. 
We assume that $g(x)=u+\lambda\sigma(x)$ for  certain $u\in\mathbb R^n$, $\lambda>0$ and $\sigma\in O(n)$, and we identify $g$ with the triplet $(u, \lambda, \sigma)\in \mathbb R^n\times (0,\infty)\times O(n)$. 

Observe that $\diam (gB)=\lambda \diam(B)$ and since $gB\in\mathcal O$ we can use inequality~(\ref{d:diametros}) to conclude that
$$ m-2\delta<\lambda \diam(B)<  m+2\delta.$$
Since $B\in\mathcal O$ too, we can use inequality~(\ref{d:diametros}) again to infer that
$$\frac{ m-2\delta}{ m+2\delta}<\lambda<\frac{ m+2\delta}{ m-2\delta}.$$
Thus, $\lambda$ lies in the compact segment $I:=\left[\frac{ m-2\delta}{ m+2\delta}, \frac{ m+2\delta}{ m-2\delta}\right].$

On the other hand, for every $b\in B$ we get that $M\geq \|b\|$ and $M\geq \|g(b)\|$. Then
\begin{align*}
    M&\geq \|g(b)\|=\|u+\lambda\sigma(x)\|\\
    &\geq \|u\|-\|\lambda \sigma (b)\|= \|u\|-\lambda\| \sigma (b)\|.
\end{align*}

From this last inequality we deduce that
\begin{align*}
    \|u\|&\leq M+\lambda\| \sigma (b)\|\leq M+\lambda M=M(1+\lambda)\\
    &\leq M\left(1+\frac{ m+2\delta}{ m-2\delta}\right).
\end{align*}

This implies that $(u,\lambda, \sigma)$ belongs to the compact set 
$$K\times I\times O(n)\subset \mathbb R^n\times (0,\infty)\times O(n)$$
where $K:=\left\{x\in\mathbb R^n\mid  \|x\|\leq M\left(1+\frac{ m+2\delta}{ m-2\delta}\right)\right\}$.
Finally, if we identify every $g\in \Gamma$ with its corresponding triplet $(u, \lambda,\sigma)$ we conclude that 
$\Gamma\subset K\times I\times O(n)$ and therefore $\overline {\Gamma}$ is compact, as desired. 
\end{proof}

\begin{theorem}\label{t:S(n) 1}
The pair $(\mathcal K^n_{i^{+}},\mathbb R^n)$ has the 
 $\Sim(n)$-finite extension property.
\end{theorem}
\begin{proof}

By Lemma~\ref{l:Sim propio}  $\Sim(n)$ is a Lie group acting properly on $\mathcal K^n_{i^{+}}$.  On the other hand, the set $C_{\Sim(n)}(\mathcal K^n_{i^{+}},\mathbb R^n)$ is nonempty (it contains, for instance, the Chebyshev point) and $\mathbb R^n$ is $\Sim(n)$-equiconnected. Then, we can use Corollary~\ref{c:main} to conclude that $(\mathcal K^n_{i^{+}},\mathbb R^n)$ has the 
 $\Sim(n)$-finite extension property, as desired.
\end{proof}

\begin{remark}
The case $i=n$ of Theorem~\ref{t:S(n) 1} can be found in \cite{Kuchment}.
\end{remark}
\begin{theorem}\label{t:S(n) 2}
For every  $1\leq i\leq n$ and $0\leq m\leq n$, the pair $(\mathcal K^n_{i^{+}}, \mathcal K^n_{m^{+}})$ has the $\Sim(n)$-finite extension property.
\end{theorem}

\begin{proof}
By the Remark~\ref{r:minkowski sum equiconnected}, the space $\mathcal K^{n}_{ m^+}$ is $\Aff(n)$-equiconnected and therefore it is $\Sim(n)$-equiconnected. Furthermore, the map $\mathcal K^{n}_{i^{+}}\to \mathcal K^{n}_{m^{+}}$ assigning to each compact convex set its circumball is well-defined, continuous, and $\Sim(n)$-equivariant (see, e.g. \cite{Maria}). Thus, $C_{\Sim(n)}(\mathcal K^{n}_{i^{+}},\mathcal K^{n}_{m^{+}})$
is non empty for every $1\leq i\leq n$ and $0\leq j\leq n$. Finally we can use Corollary~\ref{c:main} and Lemma~\ref{l:Sim propio} to conclude that 
$(\mathcal K^n_{i^{+}}, \mathcal K^n_{m^{+}})$ has the $\Sim(n)$-finite extension property.
\end{proof}

For the last theorem we denote by $\mathcal K_w^n$ the subspace of $\mathcal K^n_0$ consisting of all convex bodies  of constant width. Recall that an element $A\in \mathcal K^n_0$ is said to be of constant width if the distance between any two of its parallel support hyperplanes is constant. Equivalently, a convex body  $A\in \mathcal K^n_0$ has constant width if and only if there exists $d>0$ such that
$$A+(-A)=d\mathbb B^n=\{x\in\mathbb R^n\mid \|x\|\leq d\}.$$
From this we can easily deduce that $gA\in \mathcal K^n_w$ for every $A\in \mathcal K^n_w$ and $g\in \Sim (n)$. Thus, $\mathcal K^n_w$ is a $\Sim(n)$-invariant subspace of  $\mathcal K^n_0$. 
For more information about bodies of constant width we refer the reader to  \cite{Montejano} and \cite[Chapter 7, \S 6]{Webster}.

\begin{theorem}
For every $i\geq 1$, the pair $(\mathcal K^n_{i^+}, \mathcal K^n_w)$ has the 
 $\Sim(n)$-finite extension property.
\end{theorem}

\begin{proof}
For any two elements $A, B\in\mathcal K^n_w$ and every $t\in[0,1]$, the set $tA+(1-t)B$ belongs to $K^n_w$ (see, e.g. \cite[p.350]{NosotrosConstant}).  Thus, $K^n_w$ is $\Sim(n)$-equiconnected. 

Since $\mathcal K^n_{i^+}$ is a $\Sim (n)$-proper space (Lemma \ref{l:Sim propio}), it only rest to prove that $C_{\Sim(n)}(\mathcal K^n_{i^+}, \mathcal K^n_w)$ is nonempty. However this follows from the fact that every closed ball of positive radius belongs to $\mathcal K_w^n$. Therefore the map assigning to each $A\in \mathcal K^n_{i^+}$ its circumball belongs to $C_{\Sim(n)}(\mathcal K^n_{i^+}, \mathcal K^n_w)$.
\end{proof}

\section{Final remarks and questions}\label{s:final}

 Let $G$ be a topological group and $X$ a
 $G$-space.  
The $G$-finite extension property is related to the property of being an equivariant absolute extensor.  We recall that a $G$-space $Y$ is called an \textit{equivariant absolute neighborhood extensor} for $X$ (denoted by $Y\in G$-$\mathrm{ANE}(X)$) if, for any closed invariant subset $A\subset X$ and any equivariant map $f:A\to Y$, there exists an invariant neighborhood $U$ of $A$ in $X$ and an equivariant map $F:U\to Y$ such that $F|_{A}=f$.  If we can always take $U=X$, then we say that $Y$ is a $G$-\textit{equivariant absolute extensor} for $X$ (denoted by $Y\in G$-$\mathrm{AE}(X)$). 
If $Y$ is a $G$-$\mathrm{ANE}(X)$ ($Y\in G$-$\mathrm{AE}(X)$) for every metric $G$-space $X$, then we simply say that $Y$ is a \textit{$G$-absolute neighborhood extensor} (\textit{$G$-absolute extensor}) and we denote it by $Y\in G$-$\mathrm{ANE}$ ($Y\in G$-$\mathrm{ANE}$). We refer the reader to  \cite{equivariant theory of retracts} for more information about equivariant absolute extensors.  The reader may compare this definition with the (non equivariant) definition of an absolute extensor (see e.g. \cite[Chapter 6]{Sakai}).

Our interest in $G$-$\mathrm{AE}$'s arises from the fact that the $G$-finite extension property is a weaker property than the $G$-absolute extensor property.

\begin{remark}\label{r:AE FINITE EXTENSION}
If $X$ is a cartan $G$-space and $Y\in G$-$\mathrm{AE}(X)$ then  the pair $(X, Y)$ has the $G$-finite extensor property. Indeed, 
let  $\{x_1,\dots, x_n\}\subset X $ and $\{y_1,\dots, y_n\}\subset Y$  be two finite sets satisfying 
 $G_{x_i}\leq G_{y_i}$ for every $i\in \{1,\dots, n\}$. Since every orbit $G(x_i)$ is closed in $X$, the set $A:=\bigcup\limits_{i=1}^nG(x_i)$ is an invariant closed subset of $X$.
 If $G(x_i)\cap G(x_j)=\emptyset$  (for $i\neq j$), then the map $f:A\to Y$
defined by 
$f(gx_i)=gy_i$ is well-defined, continuous and equivariant.
 Therefore,  the definition of a $G$-equivariant absolute extensor guarantees the existence of a  continuous and $G$-equivariant map $F:X\to Y$ such that $F|_A=f$. In particular $F(x_i)=y_i$, as desired. 
\end{remark}

 After Remark~\ref{r:AE FINITE EXTENSION}  it is natural to ask under which conditions the converse implication would be true. In particular, since the pair $(\mathcal K^n_0,\mathbb R^n)$ has the $\Aff(n)$-finite extension property, it is natural to ask the following

\begin{question}\label{q:Rn AE}
 Does $\mathbb R^n\in\Aff(n)$-$\mathrm{AE}(\mathcal K^n_0)$?
\end{question}

We notice that if $G$ is a compact group acting linearly on a locally convex linear space $L$, then every complete convex invariant subset $V\subset L$ is a $G$-absolute extensor (see \cite[Theorem 2]{Antonyan 3}). This result is an equivariant version of the well-known Dugundji extension theorem.  Some other equivaraint extension theorems for compact Lie groups can be found in \cite{Antonyan 1, Convex strucuture}.

However, if $G$ is a non compact group, the problem of determining  if a $G$-space $X$ is a $G$-$\mathrm{AE}$ is not an easy task and the results found in the literature  (see e.g. \cite[Theorem 4.1]{equivariant selection}) do not provide a large class of examples of $G$-equivariant absolute extensors.  In particular, the following simple question is unknown.

\begin{question}
 Is $\mathcal K^n_0$ an $\Aff(n)$-absolute  (neighborhood) extensor?
\end{question}

For the last questions, we define for every $i\in\{0, \dots, n\}$ the family $\mathcal K^n(i)$ of all $i$-dimensional compact convex sets of $\mathbb R^n$. Namely, 
$$\mathcal K^n(i):=\mathcal K^n_{i^-}\cap\mathcal K^n_{i^+}.$$
Observe that $\mathcal K^n(0)=\mathbb R^n$ and $\mathcal K^n(n)=\mathcal K^n_0$. Thus, when $i\in\{0, n\}$ we have  that the pair $(\mathcal K^n_{0}, \mathcal K^n(i))$ has the $\Aff(n)$-finite extension property (\cite[Corollary 4.1]{Olaf} and Theorem~\ref{t:kn0 kn0}). 
However, if $i\notin \{0, n\}$, the Minkowski sum is not an $\Aff(n)$-connecting map on $\mathcal K^n(i)$. This suggests the following natural questions:

\begin{question}
 Is there an $i\in\{1,\dots, n-1\}$ such that the pair $(\mathcal K^n_0, \mathcal K^n(i))$ has the $\Aff(n)$-finite extension property? Or the $G$-finite extension property for a certain closed subgroup $G\leq \Aff(n)$?
\end{question}

\begin{question}
 For which $i, j\in\{0,\dots, n\}$ and $G\leq \Aff(n)$, does the pair $(\mathcal K^n_{j^+}, \mathcal K^n(i))$ have the $G$-finite extension property?
\end{question}

\textit{Acknowledgement}. The author want to thank the anonymous referee for the careful reading and the suggestions that helped to improve the final version of this paper.

\bibliographystyle{amsplain}

\end{document}